\documentclass{article}

\usepackage{setspace}
\usepackage{amsmath}
\usepackage{amstext}
\usepackage{amssymb}
\usepackage{graphicx}
\usepackage{amsthm}

\theoremstyle{plain}
\newtheorem{theorem}{Theorem}[section]
\newtheorem{lemma}[theorem]{Lemma}

\newtheorem{proposition}[theorem]{Proposition}
\newtheorem{refthm}{Theorem}

\theoremstyle{definition}

\newtheorem{example}[theorem]{Example}
\numberwithin{equation}{section}

\newcommand{\beq}{\begin{equation}}
\newcommand{\eeq}{\end{equation}}

\newcommand{\C}{\mathbb{C}}
\newcommand{\R}{\mathbb{R}}

\newcommand{\D}{\mathbb{D}}

\renewcommand{\Re}{\text{Re }}

\DeclareMathOperator{\sgn}{sgn}

\newcommand{\conj}[1]{\overline{#1}}

\newcommand{\bp}{\mathcal{P}}

\begin{document}

\title{Solution of Extremal Problems in Bergman Spaces Using the Bergman Projection}

\renewcommand\footnotemark{}
\author{Timothy Ferguson\footnote{Thanks to Peter Duren for his help in editing 
the manuscript, and 
to the referee for many helpful suggestions.}
}


\maketitle

\begin{abstract}
In this paper we discuss the explicit solution of certain extremal problems in 
Bergman spaces. In order to do this, we develop methods to calculate the 
Bergman projections of various functions.  As a special case, we deal with 
canonical divisors for certain values of $p$.
\end{abstract}

\section{Introduction}\label{intro}

This paper deals with linear extremal problems in Bergman spaces.  
The study of extremal problems in Bergman spaces was inspired by 
extremal problems in Hardy spaces, which
have been studied by various authors, notably by 
Macintyre and Rogosinski (see \cite{Macintyre_Rogosinski}), 
Rogosinski and Shapiro (see \cite{Rogosinski_Shapiro}), 
and S.~Ya.~Khavinson 
(see \cite{KhavinsonSYa1949} and \cite{KhavinsonSYa1951}).

Bergman space extremal problems have been studied by 
various authors, for example
in \cite{Khavinson_Stessin}, \cite{Dragan}, \cite{Ryabykh_certain_extp}, 
\cite{Hedenmalm_canonical_A2}, 
and \cite{DKSS_Pac}.  
See also the survey \cite{Beneteau_Khavinson_survey}.  
Regularity results for these problems have been studied in 
\cite{Ryabykh}, \cite{Khavinson_McCarthy_Shapiro}, 
\cite{tjf1}, and \cite{tjf2}.  
However, there are still no general methods for
finding solutions to these problems, and few explicit solutions are
available.  This is in contrast to the situation for Hardy spaces,
where a rich theory based on duality and functional analysis 
allows many extremal problems to be explicitly solved 
(see the references in the previous paragraph.)

This paper introduces methods for finding explicit solutions to certain 
extremal problems in Bergman spaces.  For example, we solve certain
minimal interpolation problems involving finding the smallest norm of
a Bergman space function when its value and the value of its first two
derivatives are specified at the origin. 
Similar results to ours are obtained in other works, for 
example 
\cite{Khavinson_Stessin}, \cite{Osipenko_Stessin}, and 
\cite{SheilSmall_2011}.
As another example, we find the function that maximizes 
the functional defined by 
$f \mapsto f^{(n)}(0) + bf(0)$ for certain values of $b$. 
The methods are based on theorems 
developed in the paper 
about the relation between the Bergman projection and extremal problems, 
as well as calculations of Bergman projections of various functions.  As 
a special case, we deal with canonical divisors, also known as contractive 
divisors, for certain $A^p$ spaces.

An analytic function $f$ in the unit disc $\D$ is said to belong to the 
Bergman space $A^p$ if 
\[  \|f\|_{A^p} = \left\{ \int_{\D} |f(z)|^p \, d\sigma(z)\right\}^{1/p} < \infty. \]
Here $\sigma$ denotes normalized area measure, so that $\sigma(\D)=1.$ 

For $1 < p < \infty,$ the dual of the Bergman space $A^p$ is isomorphic to $A^q$, 
where $1/p + 1/q = 1,$ and $k\in A^q$ represents the functional 
defined by 
\[
\phi(f) = \int_{\D} f(z)\conj{k(z)}\,d\sigma(z).
\]
Note that this isomorphism is actually conjugate-linear.  
It is not an isometry unless $p=2$, 
but  
if the functional $\phi \in (A^p)^*$ is represented by the 
function $k \in A^q$, then
\begin{equation}\label{A_q_isomorphism}
\| \phi \|_{(A^p)^*} \le \| k \|_{A^q} \le C_p \| \phi \|_{(A^p)^*}
\end{equation}
where $C_p$ is a constant depending only on $p$.

In this paper the only Bergman spaces we consider are those with
$1<p<\infty$. 
The case $p \le 1$ is more difficult because the proof of 
Theorem \ref{lin_ext_proj} fails for $p \le 1$. 
This theorem is a key result needed for our method of solving extremal 
problems. 
The proof of Theorem \ref{lin_ext_proj} relies on the boundedness of the 
Bergman projection on $L^p$, which fails for $p \le 1$.  
It also relies on H\"{o}lder's inequality, which fails for $p < 1$. 

For a given linear 
functional $\phi \in (A^p)^*$ such that $\phi \ne 0$,
we investigate the linear extremal problem of finding a function $F \in A^p$ 
with norm 
$\|F\|_{A^p} = 1$ for which
\begin{equation}\label{norm1}
\Re \phi(F) = \sup_{\|g\|_{A^p}=1} \Re \phi(g) = \| \phi \|.
\end{equation}
Such a function $F$ is called an extremal function, and 
we say that $F$ is an extremal function for a function $k \in A^q$ 
if $F$ solves problem 
\eqref{norm1} for the functional $\phi$ with kernel $k$. 
This problem has been studied by numerous authors (see the 
introduction and references for some examples). 
Note that for $p=2$, the extremal function is $F = k/\|k\|_{A^2}.$ 

A closely related problem is that of finding $f\in A^p$
such that 
$\phi(f) = 1$ and 
\begin{equation}\label{value1}
\|f\|_{A^p} = \inf_{\phi(g) = 1} \|g\|_{A^p}.
\end{equation}
If $F$ solves the problem \eqref{norm1}, then $\frac{F}{\phi(F)}$ 
solves the problem 
\eqref{value1}, and if $f$ solves \eqref{value1}, then 
$\frac{f}{\|f\|}$ solves \eqref{norm1}.
When discussing either of these problems, we always assume that $\phi$ is 
not the zero functional, in other words, that $k$ is not identically $0$. 

It is well known that the problems \eqref{norm1} and 
\eqref{value1} each have a unique solution 
when $1<p<\infty$ (see e.g.~\cite{Ryabykh}, or \cite{tjf1}, Theorem 1.4). 
Also, for every 
function $f \in A^p$ such that $f$ is not identically $0$, there 
is a unique $k \in A^q$ such that 
$f$ solves problem \eqref{value1} for $k$ 
(see e.g.~\cite{tjf1}, Theorem 3.3).  This implies that for each $F \in A^p$ with 
$\|F\|_{A^p} = 1$, there is some nonzero $k$ such that $F$ solves problem 
\eqref{norm1} for $k.$  Furthermore, any two such kernels $k$ are positive 
multiples of each other. 

The next result is an important characterization of extremal functions in 
$A^p$ for $1<p<\infty$ (see \cite{Shapiro_Approx}, p.~55). 
\begin{refthm}\label{integral_extremal_condition} 
Let $1 < p < \infty$ and let $\phi \in (A^p)^*$.   
A function $F \in A^p$ with $\|F\|_{A^p} = 1$ and 
$\Re \phi(F) > 0$ satisfies 
$$\Re \phi(F) = \sup_{\|g\|_{A^p} =1} \Re \phi(g) = \| \phi \|$$
if and only if
$$\int_{\D} h |F|^{p-1} \conj{\sgn F}  \, d\sigma = 0$$ for all 
$h \in A^p$ with $\phi(h) = 0.$  
If $F$ satisfies the above conditions, then 
$$\int_{\D} h |F|^{p-1} \conj{\sgn F}\,  d\sigma 
= \frac{\phi(h)}{\| \phi \|}$$
for all $h\in A^p.$
\end{refthm}

The following may also be found in \cite{Shapiro_Approx}, p.~55.
\begin{refthm}\label{alt_integral_extremal_condition}
Suppose that $X$ is a closed subspace of $L^p(\D)$, for $1<p<\infty$.  
Let $F \in L^p$ and suppose that for all $h \in X$, 
we have $\|F\| \le \|F + h\|.$  Then, 
\[
\int_{\D} h|F|^{p-1} \conj{\sgn F} d\sigma = 0
\]
for all $h \in X.$   
\end{refthm}

Because point evaluation is a bounded linear functional on the Hilbert space 
$A^2$, the space 
$A^2$ has a reproducing kernel $K(z,\zeta)$, called the 
\emph{Bergman kernel}, with the property that  
\beq\label{eq_intro_reproducing_formula}
f(z) = \int_{\D} K(z, \zeta) f(\zeta) \, d\sigma(\zeta)
\eeq
for all $f \in A^2$ and for all $z \in \D.$  One can show that 
$$K(z,\zeta) = \frac{1}{(1-\conj{\zeta}z)^2}.$$  Since the polynomials are 
dense in $A^1$, we have that \eqref{eq_intro_reproducing_formula} holds 
for all $f \in A^1.$  

In fact, for any $f$ in $L^1$ we many define the Bergman projection 
$\bp$ by 
\[
(\bp f)(z) = \int_{\D} \frac{f(\zeta)}{(1-\conj{\zeta}z)^2}\, d\sigma(\zeta).
\]
The Bergman projection maps $L^1$ into the space of functions analytic in $\D.$  A non-trivial 
fact is that $\bp$ also maps $L^p$ boundedly onto $A^p$ for $1<p<\infty.$  
If $p = 2,$ then $\bp$ is just the orthogonal projection of $L^2$ onto 
$A^2.$ 

The rest of this paper is organized as follows.  In section
\ref{rel_berg_ext}, we prove various theorems relating the Bergman projection 
to extremal problems.  In section \ref{calc_bp}, we calculate various 
Bergman projections.  We use these results in section 
\ref{specific_probs} to solve some extremal problems explicitly.  Lastly, 
in section \ref{canon_divisors}, we apply our results to the study of 
canonical divisors in $A^p$ when $p$ is an even integer.

\section{Relation of the Bergman Projection to Extremal Problems}\label{rel_berg_ext}
In this section we show how information about the Bergman projection can be used to 
solve certain extremal problems. 
We begin with a basic theorem that is obvious but quite useful.
\begin{theorem}\label{bp_equality} 
Suppose that $1<p<\infty$ and 
let $f \in A^p$ and $g \in L^q$, where $1/p + 1/q = 1.$ 
Then 
\[
\int_{\D} f \conj{g}\, d\sigma = 
\int_{\D} f \, \conj{\bp(g)}\, d\sigma. 
\]
\end{theorem}
\begin{proof}
The case $p=2$ follows from the fact that $\bp$ is the orthogonal projection from 
$L^2$ onto $A^2$.  The other cases follow from a routine approximation argument, using 
the fact that $\bp : L^p \rightarrow A^p$ boundedly.
\end{proof}

The next theorem gives the first indication of how the Bergman projection 
is related to extremal problems.
\begin{theorem}\label{lin_ext_proj} Suppose that $1<p<\infty$. 
Let $F \in A^p$  with $\|F\|_{A^p} = 1.$
Then $F$ is the extremal function for the functional with kernel 
\[
k = \bp(|F|^{p-1} \sgn F) \in A^q.
\]
Furthermore, if $F$ is the extremal function for some 
functional $\phi \in (A^p)^*$ with kernel 
$k \in A^q,$ then 
$$
\bp(|F|^{p-1} \sgn F) = \frac{k}{\|\phi\|}.
$$
\end{theorem} 
\begin{proof}
Consider the functional $\psi \in (A^p)^*$ that takes a function 
$f \in A^p$ to 
$$
\psi(f) = \int_{\D} f |F|^{p-1}\conj{\sgn F}\, d\sigma.
$$
This functional has norm at most $\|\, |F|^{p-1}\conj{\sgn F} \,\|_{L^q} = 
\|F\|^{p/q}_{L^p} = 1$.  But also $\psi(F) = \|F\|^p_{A^p} = 1$, so $\psi$ has norm 
exactly $1$ and $F$ is the extremal function for $\psi$.  

But from Theorem \ref{bp_equality}, it follows that 
$$
\int_{\D} f\, \conj{\bp(|F|^{p-1}\sgn F)}\, d\sigma = 
\int_{\D} f |F|^{p-1}\conj{\sgn F}\, d\sigma 
$$
for any $f \in A^p$, which means that $\bp(|F|^{p-1}\sgn F)$ 
is the kernel in $A^q$  
representing $\psi$.  This proves the first part of 
the theorem.

If $F$ is the 
extremal function for $\phi$, then $\psi$ is a 
positive scalar multiple of $\phi$ (see Section \ref{intro}.)
Since $\| \psi\| =1$ and $\psi$ is a positive scalar 
multiple of $\phi$, it must be that $\psi = \phi/\|\phi\|.$ 
But this implies that $\bp(|F|^{p-1}\sgn F) = k/\|\phi\|.$ 
\end{proof}

The next result, Theorem \ref{min_int1},
describes the relation of the Bergman projection to 
a sort of generalized minimal interpolation problem.  The problem is 
to find the function of smallest norm such that prescribed
 linear functionals acting on the function take prescribed values.  
We will first need the following lemma.  
\begin{lemma}\label{lin_alg_lemma}
Let $V$ be a vector space over $\C$, and let $\phi, \phi_1, \ldots, \phi_N$ be 
linear functionals on $V$ such that, for $v \in V,$ 
if $\phi_1(v) = \cdots = \phi_N(v) = 0$, then $\phi(v) = 0.$  Then 
$\phi = \sum_{j=1}^N c_j \phi_j$ for some constants $c_j.$ 
\end{lemma}
The statement and proof of this lemma may be found in 
\cite{Conway_Functional} in Appendix A.2 as Proposition 1.4.

\begin{theorem}\label{min_int1}
Let $1<p<\infty$ and let $\phi_1, \phi_2, \ldots, \phi_N \in (A^p)^*$ be linearly independent.  
Then a function $F \in A^p$ satisfies
$$ \|F\|_{A^p} = \inf\{\|f\|_{A^p} : \phi_1(f) = \phi_1(F), \ldots ,
\phi_N(f) = \phi_N(F) \}$$
if and only if 
$\bp(|F|^{p-1} \sgn F)$ is a linear combination of the kernels of 
$\phi_1, \ldots, \phi_N.$
\end{theorem}
Note that this theorem gives a necessary and sufficient condition for 
a function $F$ to solve the minimal interpolation problem of finding 
a function $f \in A^p$ of smallest norm such that 
$\phi_j(f) = c_j$ for $1 \le j \le N$,
where $\phi_j \in (A^p)^*$ are arbitrary linearly independent functionals 
and the $c_j$ are given constants. 
Namely, $F$ solves the 
problem if and only if $\phi_j(F) = c_j$ for $1 \le j \le N$ and 
$\bp(|F|^{p-1} \sgn F)$ is a linear combination of the kernels of 
$\phi_1, \ldots, \phi_N$.  
Note that for the case 
$1<p<\infty$,  the problem under discussion will always have 
a unique solution (see e.g.~\cite{tjf1}, Proposition 1.3).
\begin{proof}
Let $k_1, \ldots, k_N$ be the kernels of $\phi_1, \ldots, \phi_N$, 
respectively.  
Suppose that 
$$ \|F\|_{A^p} = \inf\{\|f\|_{A^p} : \phi_1(f) = \phi_1(F), \ldots ,
\phi_N(f) = \phi_N(F) \}$$
 and let 
$h$ be any non-zero 
$A^p$ function such that $\phi_1(h) = \cdots = \phi_N(h) = 0.$ 
Since there are only a finite number of the $\phi_j$, it is clear that 
such a function exists.  
Then 
$F + h$ is also in contention to solve the extremal problem, 
so $\|F\| \le \|F + h\|$.  Now Theorem \ref{alt_integral_extremal_condition} 
shows that 
\[
\int_{\D} |F|^{p-1} \conj{\sgn F}\, h\, d\sigma = 0,
\]
and so by Theorem \ref{bp_equality} 
\[
\int_{\D} \conj{\bp(|F|^{p-1} \sgn F)}\, h \,d\sigma = 0.
\]
Define 
\[
\qquad \psi(f) = \int_{\D} \conj{\bp(|F|^{p-1} \sgn F )}\, f\, d\sigma, \qquad
f \in A^p.
\]
Lemma \ref{lin_alg_lemma} now shows that $$\psi = \sum_{j=1}^N c_j \phi_j,$$ for some 
constants $c_j$,
so $\bp(|F|^{p-1} \sgn F)$ is a linear combination of $k_1, \ldots, k_n.$ 
This proves the ``only if'' part of the theorem.

Conversely, suppose $ \bp(|F|^{p-1}\sgn F)$ is a linear combination of 
$k_1, \ldots, k_n.$  Now 
\beq\label{min_int1_eq1}
\|F\|_{A^p}^p = \int_{\D} F |F|^{p-1}\conj{\sgn F}\,d\sigma = 
\int_{\D} F\,  \conj{\bp(|F|^{p-1}\sgn F)}\,d\sigma,
\eeq 
by Theorem \ref{bp_equality}.
Now let $h \in A^p$ be such that $\phi_j(h) = 0$ for  
$1 \le j \le N.$  Since  
$\bp(|F|^{p-1}\sgn F )$ is a linear combination of the $k_j$,  
equation \eqref{min_int1_eq1} gives
\[
\begin{split}
\|F\|_{A^p}^p &= \int_{\D} (F+h)\conj{\bp(|F|^{p-1}\sgn F)}\,d\sigma\\ &=
\int_{\D} (F+h) |F|^{p-1}\conj{\sgn F}\,d\sigma \\ 
&\le \|F + h\|_{A^p} \||F|^{p-1}\conj{\sgn F} \|_{A^q}\\ &= 
\|F+h\|_{A^p} \|F\|_{A^p}^{p-1}.
\end{split}
\]
Therefore,
\[
\|F\|_{A^p} \le \|F+h\|_{A^p}.
\]
Since $h$ was an arbitrary $A^p$ function with the property that 
$\phi_j(h) = 0$ for $1\le j \le N$,
this shows that $F$ solves the extremal problem in question. 
\end{proof}
When we apply this 
theorem, we will usually have each $\phi_j$ be a  
derivative-evaluation functional. By derivative-evaluation functional, we 
mean a functional defined by $f \mapsto f^{(n)}(z_0)$ for some integer 
$n \ge 0$ and some $z_0 \in \D.$ 
Note that the theorem implies that, if 
$\phi_1, \phi_2, \ldots, \phi_N \in (A^p)^*$ are linearly independent, 
then the following two statements are equivalent:\\
1. $F$ satisfies
$$ \|F\|_{A^p} = \inf\{\|f\|_{A^p} : \phi_1(f) = \phi_1(F), \ldots, 
\phi_N(f) = \phi_N(F) \}$$
but does not satisfy
$$ \|F\|_{A^p} = \inf\{\|f\|_{A^p} : \phi_{j_1}(f) = \phi_{j_1}(F), \ldots, 
\phi_{j_{M}}(f) = \phi_{j_{M}}(F) \}$$
for any proper subsequence $\{j_k\}_{k=1}^{M}$ of $1, 2, \ldots, N$.
\\
2. 
$\bp(|F|^{p-1} \sgn F)$ is a linear combination of the kernels of 
$\phi_1, \ldots, \phi_N$, and none of the coefficients in the linear 
combination is $0$.

The next theorem is a special case of Theorem 
\ref{min_int1}, with the functionals taken to be 
$\phi_j(h) = h^{(j)}(0)$, with kernels $k_j(z) = (j+1)! z^j.$ 
\begin{theorem}\label{min_int_poly1}
The function $\bp(|F|^{p-1}\sgn F)$ is a polynomial of degree at most $N$ 
if and only if
\[
\|F\|_{A^p} = \inf \{\|f\|_{A^p} : f(0) = F(0), \ldots, f^{(N)}(0)=F^{(N)}(0)\}.
\]
It is a polynomial of degree exactly $N$ 
if and only if $N$ is the smallest integer such that the above
conditions holds.
\end{theorem}

The next theorem relates the generalized minimal interpolation problems 
we have been discussing with linear extremal problems. 
\begin{theorem} 
Let $\phi_1, \ldots, \phi_N$ be linearly independent elements of 
$(A^p)^*$ with kernels 
$k_1, \ldots, k_N$ respectively, and let 
$F\in A^p$ with $\|F\|_{A^p} = 1$.
Then the functional for which $F$ is the extremal function has as its kernel 
a linear combination of 
the $k_j$ if and only if 
\[
\|F\|_{A^p} = \inf \{\|f\|_{A^p} : \phi_1(f) = \phi_1(F), \ldots, 
\phi_N(f) = \phi_N(F) \}.
\]
\end{theorem}
This follows from Theorems \ref{lin_ext_proj} and \ref{min_int1}.
Recall that although there is no unique functional for which 
$F$ is the extremal function, such a functional is unique up to a positive 
scalar multiple, which does not affect whether its kernel
is a linear combination of the $k_j$. 

One direction of this theorem,  
the fact that if $F$ is the extremal 
function for some kernel which is a linear combination of the $k_j$, then 
$F$ solves the stated minimal interpolation problem, is easy to prove 
directly. The proof is as follows. 
Let $F$ be the extremal function for the functional $\phi$, 
which we assume to have kernel $k = \sum_{j=1}^N a_j k_j.$ 
Then 
$$ \|F\|_{A^p} = \inf\{\|f\|_{A^p} : \phi(f) = \phi(F)\}.$$
But if some function $G$ in $A^p$ satisfies  
$\phi_j(F) = \phi_j(G)$ for all $j$ with $1\le j \le N$, 
then $\phi(G) = \phi(F),$ which implies that 
$\|F\|_{A^p} \le \|G\|_{A^p}$.  This implies that $F$ satisfies 
$$
\|F\|_{A^p} = \inf \{\|f\|_{A^p} : \phi_1(f) = \phi_1(F), \ldots, 
\phi_N(f) = \phi_N(F) \}.
$$

\section{Calculating Bergman Projections}\label{calc_bp}

Now that we have explored the relation between the Bergman projection and 
solutions to extremal problems, we will calculate the Bergman projection 
in various cases.

\begin{proposition}\label{bp_monomial}
Let $m$ and $n$ be nonnegative integers.  Then 
\[
\bp(z^m \conj{z}^n) = \begin{cases}
\frac{m-n+1}{m+1}z^{m-n}, &\text{ if $m \ge n$,}\\
0, &\text{ if $m < n$}. \end{cases}
\]
\end{proposition}
This is Lemma 6 in Chapter 2 of \cite{D_Ap}.

The next theorem is very helpful in calculating the Bergman projection of 
the kernel of a derivative-evaluation functional times the conjugate of 
an $A^p$ function.
\begin{theorem}\label{bp_f_conjg}
Let $1<q_1,q_2\le\infty$.  Let $p_1$ and $p_2$ be the conjugate exponents of 
$q_1$ and $q_2$.  Let $$\frac{1}{q} = \frac{1}{q_1}  + \frac{1}{q_2}$$ and 
suppose that $1<q<\infty.$  Let $p$ be the conjugate exponent of $q$. 
Suppose that $k \in A^{q_1}$ and that $g \in A^{q_2}.$  Let 
the functional 
$\psi$ be defined by $\psi(f) = \int_{\D} f \conj{k}\, d\sigma$ for all 
$f \in A^{p_1}.$  
Then
$\bp(k\conj{g})$ is the kernel of the functional 
$\phi \in (A^{p})^*$ defined by  
$$ \phi(f) = \psi(fg), \qquad f\in A^p.$$
\end{theorem}

\begin{proof}
First note that $1/p + 1/q_1 + 1/q_2 = 1$, so if $f \in A^p,$ 
then $fg \in A^{p_1}$ and the definition of $\phi$ makes sense. 
Now observe that 
$$ \phi(f) = \psi(fg) = \int_{\D} fg \conj{k}\, d \sigma.$$
By Theorem \ref{bp_equality}, this equals
\[
 \int_{\D} f \conj{\bp\left(\conj{g} k \right)}\, d\sigma. 
\qedhere
\]
\end{proof}

We will study the kernels of various derivative-evaluation functionals.  
Evaluation at the origin is somewhat different 
and simpler than evaluation elsewhere, so we deal with it 
first.

\begin{theorem}\label{kernel_derivative_0}
The kernel for the functional $f \mapsto f^{(n)}(0)$ is $(n+1)! z^n.$  If 
$g \in A^p$ then 
\[ \bp(z^n \conj{g(z)}) = \sum_{j=0}^n \conj{\frac{g^{n-j}(0)}{(n-j)!}}
\frac{j+1}{n+1} z^{n}.
\]
\end{theorem}
\begin{proof}
The first statement can be verified by evaluating 
\[ \int_{\D} f(z) \conj{z}^n d\sigma(z) \]
when $f$ is written as a power series.  The second part follows from 
Proposition \ref{bp_monomial}.  To see this, note that by the first part 
of the theorem,
$\bp(z^n \conj{g(z)})$ is the kernel for the functional taking $f \in A^p$ 
to 
\[
\frac{1}{(n+1)!}(fg)^{(n)}(0) = 
\frac{1}{(n+1)!} \sum_{j=0}^n \binom{n}{j} f^{(j)}(0) g^{(n-j)}(0),
\]
which has kernel 
\[\sum_{j=0}^n \conj{\frac{g^{n-j}(0)}{(n-j)!}}
\frac{j+1}{n+1} z^{j}. \qedhere
\]
\end{proof}

We will now deal with the function $1/(1-\conj{a}z)^n$, for $n \ge 2$.
\begin{proposition}\label{kernel_f_(n)}
The kernel for the functional $f \mapsto f^{(n)}(a)$ is 
$$
\qquad \qquad \qquad \qquad \frac{(n+1)!z^{n}}{(1-\conj{a}z)^{n+2}}, \qquad \qquad n=0,1,2,\ldots
$$
\end{proposition}
\begin{proof}
We know that 
$$f(a) = \int_{\D} \frac{1}{(1-a\conj{z})^2}f(z)\, d\sigma(z).$$
Differentiation $n$ times with respect to 
$a$ gives the result.  
\end{proof}

\begin{proposition}\label{(1-az)n+2_kernel}
For each $a \in \D$ with $a \ne 0$, there are numbers $c_0, \ldots, c_n$ with $c_n \ne 0$ such 
that 
the function $$\frac{1}{(1-\conj{a}z)^{n+2}}$$ is the kernel for the 
functional $f \mapsto c_0 f(a)+ c_1 f'(a) + \ldots + 
c_n f^{(n)}(a).$   
\end{proposition}
\begin{proof}
We will proceed by induction.  The claim is true for $n=0$ by the 
reproducing property of the Bergman kernel function.  
For general $n$, we may write the partial fraction expansion 
$$ \frac{z^n}{(1-\conj{a}z)^{n+2}} = \sum_{j=0}^{n+2} \frac{b_j}{(1-\conj{a}z)^j},
$$
for some complex numbers $b_j$.  
Thus, 
$$
z^n = \sum_{j=0}^{n+2} b_j (1-\conj{a}z)^{n+2-j}.$$
Differentiating both sides $n+1$ times with respect to $z$ gives
$$0 = b_1 (-\conj{a})^{n+1}(n+1)! + b_0(-\conj{a})^{n+1}(n+2)!
(1-\conj{a}z).$$
Since this holds for all $z$, it follows that $b_0 = b_1 = 0.$ 
Since $ z^n/(1-\conj{a}z)^{n+2}$ has a pole of order $n+2$ at 
$1/\conj{a}$, we see that $b_{n+2} \ne 0.$ 
Therefore,
$$\frac{1}{(1-\conj{a}z)^{n+2}} = \frac{1}{b_{n+2}}\left(\frac{z^n}{(1-\conj{a}z)^{n+2}} - 
\sum_{j=2}^{n+1} \frac{b_j}{(1-\conj{a}z)^j}\right).$$
Note that the first term of the right side of the above equation is 
the kernel for the functional $f \mapsto (1/(n+1)!)f^{(n)}(a)$.  
Also, each term in the sum 
$\sum_{j=2}^{n+1} \frac{b_j}{(1-\conj{a}z)^j}$
is the kernel for a linear functional taking 
each function $f$ to some linear combination of $f(a), f'(a), \ldots,$ and 
$f^{(n-1)}(a),$ by the induction hypothesis.  This proves the proposition. 
\end{proof}

\begin{proposition}\label{general_proj_zero}
Let $g \in A^p$ for $1<p<\infty,$ and let $a \in \D$ with 
$a \ne 0$. Suppose 
$g$ has a zero of order $n$ at $a$. 
Let $N \ge 0$ be an integer. Then 
$$\bp\left(\frac{1}{(1-\conj{a}z)^{N+2}} \conj{g(z)}\right) = 
\sum_{k=0}^{N-n}
C_k \frac{1}{(1-\conj{a}z)^{k+2}}$$
for some complex constants $C_k$ depending on $g^{(m)}(a)$ for 
$0 \le m \le N$.
\end{proposition}
\begin{proof}
The projection 
$$\bp\left( \frac{1}{(1-\conj{a}z)^{N+2}}\, \conj{g(z)} \right)$$
is the kernel associated with the functional 
$$f \mapsto \sum_{j=0}^N b_j (fg)^{(j)}(a)$$
for some constants $b_j$, with $b_N \ne 0$, by the previous proposition
and Theorem \ref{bp_f_conjg}. 
But 
\[
\sum_{j=0}^N b_j (fg)^{(j)}(a) = 
\sum_{j=0}^N \sum_{k=0}^j b_j \binom{j}{k} f^{(k)}(a) g^{(j-k)}(a).
\]
Since $g^{(j)}(a) = 0$ for $0 \le j < n,$ all terms in the sum with 
$j-k < n$ are $0$.  But this means that the only non-zero terms in the 
sum occur when $k \le j-n$, so that $k \le N-n.$  
Now, set 
\[
B_k = \sum_{j = k + n}^{N} b_j \binom{j}{k} g^{(j-k)}(a),
\]
so
\[
\sum_{j=0}^N b_j (fg)^{(j)}(a) = \sum_{k=0}^{N-n} B_k f^{(k)}(a).
\]
But the kernel associated to $\sum_{k=0}^{N-n} B_k f^{(k)}(a)$ is 
\[
\sum_{k=0}^{N-n} B_k \frac{(k+1)!z^k}{(1-\conj{a}z)^{k+2}}. 
\]
As in the proof of Theorem \ref{(1-az)n+2_kernel}, we may show that 
\[
\frac{z^k}{(1-\conj{a}z)^{k+2}} = \frac{c_{k2}}{(1-\conj{a}z)^2}   
  + \frac{c_{k3}}{(1-\conj{a}z)^3} + \cdots +   
     \frac{c_{k,k+2}}{(1-\conj{a}z)^{k+2}}
\]
for some constants $c_{k2}, \ldots, c_{k,k+2}$.
Thus we may write 
\[
\sum_{k=0}^{N-n} B_k \frac{(k+1)!z^k}{(1-\conj{a}z)^{k+2}} = 
\sum_{k=0}^{N-n} C_k \frac{1}{(1-\conj{a}z)^{k+2}}\]
for some constants $C_k$, depending on $g^{(m)}(a)$ for $0 \le m \le N$. 
\end{proof}

We will now deal with the function $1/(1-\conj{a}z).$ Since the functional 
with kernel
$1/(1-\conj{a}z)^{n+2}$ involves differentiation of order $n$, it seems 
reasonable that the functional with kernel $1/(1-\conj{a}z)$ involves 
integration.  This is indeed the case. 
\begin{proposition}\label{kernel_for_integration}
The function $$1/(1-\conj{a}z)$$ is the kernel for the functional defined 
on $A^p$ for $1<p<\infty$ by 
$$f \mapsto \frac{1}{a} \int_0^a f(z)\, dz.$$
\end{proposition}
\begin{proof}
Since $$\frac{1}{1-\conj{a}z} = \sum_{n=0}^\infty (\conj{a}z)^n,$$ it 
follows that 
\begin{equation}\label{int_1/(1-conj(a)z)z^m}
\int_{\D}\frac{z^m}{1-a\conj{z}} \, d\sigma = \sum_{n=0}^\infty 
\int_{\D}(a\conj{z})^n z^m \, d\sigma = 
a^m \int_{\D} |z|^{2m}\, d\sigma = 
\frac{a^m}{m+1}.
\end{equation}
The change in the order of integration and summation is justified by the 
fact that the sum converges uniformly in $\overline{\D}.$  
Now let $f \in A^p$ and write 
$f(z) = \sum_{m=0}^\infty b_m z^m.$  
Define $$F(z) = \frac{1}{z} \int_0^z f(\zeta)\, d\zeta 
 = \sum_{m=0}^\infty \frac{b_m}{m+1}z^m.$$
Therefore, by equation \eqref{int_1/(1-conj(a)z)z^m},
$$\int_{\D}\frac{1}{1-a\conj{z}}\, f(z)\, d\sigma = 
\int_{\D}\frac{1}{1-a\conj{z}} \left(\sum_{m=0}^\infty b_m z^m\right) d\sigma
= \sum_{m=0}^\infty b_m \frac{a^m}{m+1} = F(a).$$
The interchange of the order of integration and summation is justified by 
the fact the partial sums of the Taylor series for $f$ approach 
$f$ in $A^p.$  
\end{proof}

The following theorem is quite useful for determining what form certain
Bergman projections have.
\begin{theorem}\label{dpeval_conj_proj}
For $1 \le n \le N$, let $d_n$ be a nonnegative integer and let $z_n \in \D$.
Let $k$ be analytic and a  linear combination of the 
kernels of the functionals 
given by $f \mapsto f^{(d_n)}(z_n)$.
Let $g \in A^p$ for $p > 1.$  
Then $\bp(k\conj{g})$ is in the linear span of the set of all the kernels 
of functionals defined by 
$f \mapsto f^{(m)}(z_n)$, where $m$ is an integer with $0 \le m \le d_n$
and $n$ is an integer with $1 \le n \le N$. 
\end{theorem}
\begin{proof}
Let $k = \sum_{n=1}^N a_n k_n$, where $k_n$ is the kernel for the functional
$f \mapsto f^{(d_n)}(z_n)$. 
Then by Theorem \ref{bp_f_conjg}, $\bp(k_n \conj{g})$ is the kernel 
of the functional
$$f \mapsto (fg)^{(d_n)}(z_n) = \sum_{j=0}^{d_n} \binom{d_n}{j} 
f^{(j)}(z_n) g^{(d_n-j)}(z_n).$$
But this functional is a linear combination of functionals of the form
$$ f \mapsto f^{(m)}(z_n),$$ where $0 \le m \le d_n.$ 
\end{proof}

Due to their relation with extremal problems, we are often concerned with 
projections of the form 
$\bp(F^{p/2}\conj{F}\hbox{}^{(p/2)-1})$, where $F$ is an analytic function. 
This is well defined because 
\[ F^{p/2} \conj{F}\hbox{}^{(p/2)-1} = |F|^p/\conj{F} 
= |F|^{p-1} \sgn F.\]
The following theorems deal with this situation.
\begin{theorem}\label{dpeval_Fp2_answer} 
Let $1<p<\infty$, and let $F \in A^p$.  Furthermore, suppose 
if $p<2$ that $F^{(p/2)-1} \in A^{p_1}$ for some $p_1 > 1$. 
Also, suppose that $F^{p/2}$ is analytic and is a linear 
combination of the kernels $k_n$ corresponding to the functionals  
$f \mapsto f^{(d_n)}(z_n)$ for some integers $d_n$ and some points 
$z_n \in \D$, where $1 \le n \le N$, and where $N$ is an integer. 
Then $F$ satisfies
\[ \begin{split}
\|F\|_{A^p} = \inf \{\|f\|_{A^p} : f^{(m)}(z_n) = 
F^{(m)}(z_n)  &\text{ for all 
$n$ and $m$ such that}\\ &\text{ \hspace{3ex} $1 \le n \le N$ and 
$1 \le m \le d_n$}\}.
\end{split} \]
\end{theorem}
\begin{proof}
This follows from Theorems \ref{dpeval_conj_proj} and \ref{min_int1}. 
\end{proof}

The following theorem is a consequence of Theorem \ref{dpeval_conj_proj}.  
It can also be proved by using Taylor series and Proposition \ref{bp_monomial}.
\begin{theorem}\label{poly_conj_proj}
Let $f$ be a polynomial of degree at most $N$ 
and let $g \in A^p$ for some $p > 1.$  Then 
$\bp(f\conj{g})$ is a polynomial of degree at most $N$. 
\end{theorem}

Using this theorem and Theorem \ref{min_int_poly1}, we immediately get the 
following result.
\begin{theorem}\label{poly_Fp2_answer}
Suppose that $F \in A^p$ and $F^{p/2}$ is a polynomial of degree $N$.
Furthermore, if $p<2$ suppose that $F^{(p/2)-1} \in A^{p_1}$ for some 
$p_1 > 1$.  
Then   
\[
\|F\|_{A^p} = \inf \{\|f\|_{A^p} : f(0) = F(0), \ldots, f^{(N)}(0)=F^{(N)}(0)\}.
\]
\end{theorem}
Note that  
$F^{p/2}$ can be a polynomial only if $F$ is nonzero in $\D$ or 
$p/2$ is rational and 
all the zeros of $F$ in $\D$ are of order a multiple of $s$, where 
$r/s$ is the reduced form of $p/2$. 
If $p$ is an even integer, this poses no restriction. 
Because of this, the case where $p$ is an even integer is often
easier to work with.

\section{Solution of Specific Extremal Problems}\label{specific_probs}

We will now discuss how to solve some specific minimal interpolation problems.
Since we are dealing with the powers $p/2$ and $2/p$, neither of which 
need be an integer, we will have to take care in our calculations.  We will
introduce a lemma to facilitate this.  The lemma basically says that if 
$f$ and $g$ are analytic functions nonzero at the origin, and if 
$f^{(n)}(0) = (g^p)^{(n)}(0)$ for all $n$ such that $0 \le n \le N,$ then 
$(f^{1/p})^{(n)}(0) = g^{(n)}(0)$ for all $n$ such that 
$0 \le n \le N.$  

To state the lemma 
we first need to introduce some notation.  Suppose that the constants 
$c_0, c_1, \ldots, c_N$ are given and that $c_0 \ne 0$,  and let 
$h(z) = c_0 + c_1 z + \cdots + c_N z^N.$ 
Suppose that $a = c_0^p$ for some 
branch of the function $z^p$.  
Let $U$ be a neighborhood of the origin such that $h(U)$ is contained in 
some half plane whose boundary contains the origin, and such that 
$0 \not \in h(U)$. Then we can define $z^p$ so 
that 
it is analytic in $h(U)$ and so that $c_0^p = a.$  
We let 
$\beta_j^p(a; c_0, c_1, \ldots, c_j)$ denote the $j^{th}$ derivative of $h(z)^p$ 
at $0$.  

Note that because of the chain rule for differentiation, 
$\beta_j^p$ only depends on $j$, the constants $c_0, \ldots, c_j$, and the 
numbers  $p$ and $a$.  For the same reason, 
the value of $\beta_j^p$ is the same if we replace the function $h$ in the 
definition of $\beta_j^p$ 
by any function $\widetilde{h}$ analytic near the origin such that 
$\widetilde{h}^{(j)}(0) = c_j$ for $1 \le j \le N$.   
\begin{lemma}
Let $c_0, c_1, \ldots, c_N$ be given complex numbers, and let $p$ be a 
real number.  
Suppose that $c_0 \ne 0$, and let $a_0 = c_0^p$, for some 
branch of $z^p$.  Then
$$c_j = \beta_j^{1/p}\left(c_0; \beta_0^p(a_0;c_0), \beta_1^p(a_0;c_0, c_1), \ldots, 
\beta_j^p(a_0;c_0, \ldots, c_j)\right).$$
\end{lemma}
\begin{proof}
Let 
$a_j = \beta_j^p(a_0; c_0, \ldots, c_j)$ and 
$b_j = \beta_j^{1/p}(c_0; a_0, \ldots, a_j).$  Then $b_0=c_0.$

Now let $f(z) = \sum_{j=0}^N \frac{c_j}{j!} z^j$.  Then $f^{(j)}(0) = c_j$ 
for $0 \le j \le N$. 
Let $U$ be a neighborhood of $0$ such that there exist $r_0 > 0$ and 
$\theta_0 \in \R$ such that 
\[ 
f(U) \subset \left\{re^{i\theta}: r_0 < r \text{ and } 
\theta_0 - \frac{\pi}{2p} < \theta < \theta_0 + \frac{\pi}{2p}\right\}.
\]
Then $z^p$ can be defined as an analytic function in $f(U)$.  
Furthermore, the set
$V = (f(U))^p$ does not contain $0$ but is contained in some half plane, 
so $z^{1/p}$ can be defined as an analytic function in $V$ so that it is the 
inverse of the function $z^p$ defined in $f(U)$.  

Now define $g(z) = (f(z))^p$ for $z \in U.$  
Then 
$g^{(j)}(0) = a_j$ and $g^{1/p}(0) = c_0$,  so
$(g^{1/p})^{(j)}(0) = b_j$ for $0\le j \le N$.  
But $g(z)^{1/p} = f(z)$ for $z \in U$, so $b_j = c_j$ for $0 \le j \le N$. 
\end{proof}

We will now use the lemma to solve a specific extremal problem in certain 
cases. 
\begin{theorem}\label{nonzero_min_int1}
Let 
$c_0, \ldots, c_N$ be given 
complex numbers, and assume that $c_0 \ne 0$.
Suppose that $F \in A^p$, and $F^{(j)}(0) = c_j$ for $0 \le j \le N$, and 
\[
\|F\|_{A^p} = \inf \{\|g\|_{A^p} : g(0) = c_0, \ldots, g^{(N)}(0)=c_N\}.
\]
Let $a_0 = c_0^{p/2}$ for some branch of 
$z^p$. 
Define  
\[
f(z) = \sum_{j=0}^N \frac{\beta_j^{p/2}(a_0; c_0, \ldots, c_j)}{j!} z^j,
\]
where the $\beta_j^{p/2}$ are defined as in the beginning of this section. 
Suppose that $f^{1-(2/p)} \in A^{p_1}$ for some $p_1 > 1$. 
Also, suppose that $f$ has no zeros in $\D.$ 
Thus we may define $f^{2/p}$ so that it is analytic in $\D$ and 
so that $f^{2/p}(0) = c_0.$  
Then 
$$F = f^{2/p}.$$
The same result also holds if $p$ is rational, $2/p=r/s$ in lowest form, and 
every zero of $f$ has order a multiple of $s$. 
\end{theorem}
\begin{proof}
Note that $f^{2/p}$ is analytic in $\D$ since we have assumed $f$ 
has no zeros in $\D$, or that $p$ is rational and $2/p=r/s$ in 
lowest form and  
$f$ has only zeros whose orders are 
multiples of $s$.  Also, $f(0) = a_0$, so we may define $f^{2/p}$ so that 
$f^{2/p}(0) = c_0$. 
The $j^{th}$ derivative of $f^{2/p}$ at $0$ is 
$$\beta_j^{2/p}\left(c_0; \beta_0^{p/2}(a_0;c_0), \ldots, 
\beta_j^{p/2}(a_0;c_0, \ldots, c_j)\right)$$
for $0 \le j \le N$, which equals $c_j$ by the lemma. 
Thus $f^{2/p}$ is in contention to solve the extremal problem.

But
\[ 
\bp\left(\frac{|f^{2/p}|^p}{\,\conj{f}^{2/p}}\right) = 
\bp(f \conj{f}^{1-(2/p)})
\]
 is a polynomial of 
degree at most $N$ by Theorem 
\ref{poly_conj_proj}, so by Theorem \ref{poly_Fp2_answer}
we find $F = f^{2/p}$. 
\end{proof}

To apply this theorem, we need to show that $f$
 has no zeros in the unit disc, or has only zeros of suitable orders 
if $p$ is rational.  Then, as long as $f^{1-(2/p)} \in A^{p_1}$ 
for some $p_1 >1$, we have that $f^{2/p}$ is the extremal 
function.  Note that we do not need to know anything about the zeros of the 
extremal function itself to apply the theorem, but only about the zeros of $f$.  

Also note that if $f$ has no zeros in the unit disc, 
this theorem implies that the extremal function $F = f^{2/p}$ also has no 
zeros in the unit disc. It can also be shown that if $F$ has no zeros, then 
$F$ must equal $f^{2/p}$.  
This follows from \cite{Khavinson_Stessin}, Theorem B.  It also follows from 
the work on extremal problems in Bergman space posed over non-vanishing 
functions found in \cite{Khavinson_nonvanishing}, 
\cite{SheilSmall_2011}, \cite{SheilSmall_solution2012}, and 
\cite{SheilSmall_general2012}. 
The case where the extremal function has zeros is more challenging and 
not as well understood.  See Example \ref{one_zero_A4_example} for a 
problem in which the extremal function has one zero.

\begin{example} The solution to the minimal interpolation problem in 
$A^p$ with 
$f(0) = 1$ and $f'(0)=c_1$ is 
\[
F(z) = \left(1 + \frac{p}{2}c_1 z \right)^{2/p},
\]
provided that $|c_1| \le \frac{2}{p}$ (or $p=2$).
This is because $\beta_0^{p/2}(1;1, c_1) = 1$ and 
$\beta_1^{p/2}(1;1, c_1) = (p/2)c_1.$ 
For example, if $p = 4$ and $c_1 = \frac{1}{2}$, then 
\[
F(z) = (1+z)^{1/2}.
\] 
\end{example}
The above problem is also solved in \cite{Osipenko_Stessin} in more 
general form.  The solution to the extremal problem in the next 
example is more difficult. We do not know if it has been stated explicitly 
before, although in the case in which the extremal 
function has no zeros it does follow from 
\cite{Khavinson_Stessin}, Theorem B, 
or from the results in \cite{SheilSmall_2011}, 
if it is assumed \emph{a priori} that 
the extremal function has no zeros. 

\begin{example} The solution to the minimal interpolation problem in 
$A^p$ for $1<p<\infty$ with $F(0)=1$, and $F'(0) = c_1$, and $F''(0) = c_2$ is 
\[
F(z) = \left\{ 1 + (p/2)c_1z + [(p(p-2)/4)c_1^2 + (p/2)c_2]z^2 \right\}^{2/p},
\]
provided that the quadratic polynomial under the $2/p$ exponent in the 
equation for $F$ has no zeros in $\D$.  The solution is the same 
if $p=4$ or $4/3$ and the quadratic polynomial has a repeated 
root in the unit disc.  (The solution is also the same in the case 
$p=2$, no matter where the polynomial has roots). 
\end{example}

Linear extremal problems tend to be more difficult to solve than minimal 
interpolation problems involving derivative-evaluation functionals, 
because values of a function $f$ and its derivatives are generally easier to 
calculate than $\bp(|f|^{p-1}\sgn f)$.   Nevertheless it is possible to 
solve some linear extremal problems explicitly by the methods in this 
paper.  Here is one example. 
\begin{theorem}
Let $N \ge 1$ be an integer, let $1<p<\infty$, and let $b \in \C$ satisfy 
\[
|b| \ge 1 + \frac{1}{N+1}\left(1-\frac{2}{p}\right),
\]
and define
\[
a = 
\frac{|b|+\sqrt{|b|^2-\frac{4}{N+1}\left(1-\frac{2}{p}\right)}}{2}
\sgn b.
\]
Then the solution to the extremal problem in $A^p$ with kernel 
$z^N + b$ is  
\[ 
F(z) = 
\sgn(a^{1-(2/p)})
\frac{(z^N + a)^{2/p}}{\left(|a|^2 + 1/(N+1)\right)^{1/p}}.
\]
\end{theorem}
In the above expression for $F(z)$, 
the branch of $(z^N+a)^{2/p}$ may be chosen arbitrarily,  
but the value of $\sgn(a^{1-(2/p)})$ 
must be chosen consistently with this choice.
Note that the functional associated with the kernel $z^N + b$ 
is 
\[ \phi(f) = bf(0) + (1/(N+1)!)f^{(N)}(0). \]
Also, observe that the hypothesis of the theorem 
holds for all $N$ and $p$ if $|b| \ge 3/2$. 

\begin{proof}
The condition 
\[
|b| \ge 1 + \frac{1}{N+1}\left(1-\frac{2}{p}\right)
\]
implies that $|a| \ge 1$, so that 
$z^N + a \ne 0$ in $\D$ and $F$ is an analytic function.  
Note that 
\[
\|(z^N + a)^{2/p}\|_{A^p}^p = 
\int_{\D} |z^N + a|^2\, d\sigma = 
\int_{\D} (z^N + a) \conj{(z^N + a)}\, d\sigma = 
|a|^2 + \frac{1}{N + 1}.
\]
Thus, $\|F\|_{A^p} = 1$.

Now   
$$((z^N + a)^{2/p})^{p/2-1} = 
a^{1-2/p}+\left(1-\frac{2}{p}\right)a^{-2/p}z^N + O(z^{2N}),
$$
where we choose branches so that $((z^N+a)^{2/p})^{p/2} = z^N + a.$ 
We calculate that 
\[
\begin{split}
&\quad\;\bp\left( |z^N+a|^{p-1} \sgn(z^N + a) \right)  \\
&=\bp\left( (z^N + a) \conj{(z^N + a)}^{1-2/p}\right) \\ 
&=\bp\left( (z^N + a) \left( \conj{a^{1-2/p}} + \left(1-\frac{2}{p}\right)\conj{a^{-2/p}}\,
\conj{z^N} + O(\conj{z^{2N}})\right) \right).
\end{split}
\]
But by Proposition \ref{bp_monomial}, this equals 
\[
\begin{split}
&\quad\; \bp\left[ (z^N + a) \left( \conj{a^{1-2/p}} + \left(1-\frac{2}{p}\right)
\conj{a^{-2/p}}\,
\conj{z^N}\right) \right]\\
&=a \conj{a^{1-2/p}} + \frac{1}{N+1}\left(1-\frac{2}{p}\right)
\conj{a^{-2/p}}
+ \conj{a^{1-(2/p)}} z^N\\
&= 
\conj{a^{1-(2/p)}}\left(
z^N + a + \frac{1}{N+1}\left(1-\frac{2}{p}\right)\conj{a}\,^{-1} \right)\\
&= \conj{a^{1-(2/p)}}(z^N + b).
\end{split}
\]

Thus, 
\[
\begin{split}
&\quad\; \bp\left\{ \left |\sgn(a^{1-(2/p)})(z^N+a)\right|^{p-1}
\sgn\left( \sgn(a^{1-(2/p)})(z^N + a) \right) \right\}\\ 
&=\conj{a^{1-(2/p)}} \sgn(a^{1-(2/p)}) (z^N + b) \\ 
&=|a^{1-(2/p)}\!|\,(z^N + b).
\end{split}
\]

Therefore, 
\[
\bp(F^{p/2}\conj{F^{(p/2)-1}}) = 
|a^{1-(2/p)}| \frac{z^N + b}{\left(|a|^2 + 1/(N+1)\right)^{(p-1)/p}}.
\] 
Since $\|F\|_{A^p} = 1$, Theorem \ref{lin_ext_proj} shows that $F$ is the 
extremal function for the kernel on the right of the above equation. 
But that kernel is a positive scalar multiple of $k$, 
so $F$ is also the extremal 
function for $k$. 
\end{proof}

\begin{example}\label{one_zero_A4_example}
Let $a \in \D\setminus\{0\}$ and let $b, c \in \C$.  Consider the function
\[
f(z) = \frac{1}{a}\frac{a-z}{1-\conj{a}z}\left( 1+bz+cz^2 \right)^{1/2},
\]
where we assume that $1+bz+cz^2$ has no zeros in $\D$, or 
a double zero in $\D$ (so that 
$f$ is analytic).  We choose the branch of this function so that 
$f(0)=1$.  Then a calculation shows that 
\[
\begin{split}
f'(0) &= \frac{1}{a} \left(\frac{ab}{2}+|a|^2-1\right), \text{ and}\\
f''(0) &= \frac{1}{a} \left( |a|^2 b +2a\conj{a}^2 + ac - 2\conj{a} - b 
                             -\frac{ab^2}{4} \right).\\
\end{split}
\]
Another calculation shows that the residue of $f^2$ about $\conj{a^{-1}}$ is 
equal to
\[
|a|^{-4} \conj{a}^{-3}
\left[
(|a|^2-1)^2 (2c + \conj{a}b) - 
2 (|a|^2 -1)(\conj{a}^2 + c + \conj{a}b) 
\right].
\]

Now, suppose that $v_1$ and $v_2$ are complex numbers, and that we want to 
solve the minimal interpolation problem of finding $F \in A^4$ such that 
$F(0)=1, F'(0)=v_1, F''(0)=v_2$, and with $\|F\|_{4}$ as small as possible.
If we can find numbers $a$, $b$, and 
$c$ so that $(1+bz+cz^2)$ has no zeros in $\D$, or 
a repeated zero, and so that
\[
\begin{split}
v_1 &= \frac{1}{a} \left(\frac{ab}{2}+|a|^2-1\right)\\
v_2 &= \frac{1}{a} \left( |a|^2 b +2a\conj{a}^2 + ac - 2\conj{a} - b 
                             -\frac{ab^2}{4} \right)\\
0 &=  (|a|^2-1)^2 (2c + \conj{a}b) - 
       2 (|a|^2 -1)(\conj{a}^2 + c + \conj{a}b)
\end{split}
\]
then 
\[
F(z) = f(z) = \frac{1}{a} \frac{a-z}{1-\conj{a}z} 
                \left( 1 + bz + cz^2 \right)^{1/2}.
\]
To see this, note that in this case 
\[
f(z)^2 = a_1 z^2 + a_2 z + a_3 + \frac{a_4}{(1-\conj{a}z)^2}
\]
for some constants $a_1, \ldots, a_4.$
Then 
\[ \bp\left[ \left(a_1 z^2 + a_2 z + a_3\right) \conj{f(z)} \right] \]
will be a polynomial of degree at most two by 
Theorem \ref{poly_conj_proj}.  Also, since $f(a) = 0$, we have 
\[ 
\bp \left[ \frac{a_4}{(1-\conj{a}z)^2} \conj{f(z)}\right] = 0
\]
by Proposition \ref{general_proj_zero}.
Thus, $\bp(f^2 \conj{f})$ is a polynomial, and so $f$ solves the extremal 
problem in question, by Theorem \ref{min_int_poly1}.

\end{example}

\section{Canonical Divisors}\label{canon_divisors}

We will now discuss how our previous results apply to canonical divisors.  
These divisors are the Bergman space analogues of 
Blaschke products.  
They were first introduced in the $A^2$ case in 
\cite{Hedenmalm_canonical_A2}, and were further studied for general
$p$ in \cite{DKSS_Pac} and \cite{DKSS_Mich}.  
The formula for a canonical divisor with one zero is well known, see
for example \cite{D_Ap}.
In 
\cite{Hansbo}, a formula was obtained for canonical divisors with 
two zeros, as well as with more zeros under certain symmetry
conditions on the zeros.  
In \cite{Luciano_Naris_Schuster}, a method is given for finding the 
canonical divisor in $A^2$ for an arbitrary finite zero set.  
In \cite{MacGregor_Stessin}, a fairly explicit formula for 
canonical divisors is obtained for general $p$. 
In this section, we discuss how the methods
of this paper apply to the problem of finding canonical divisors in
the case where $p$ is an even integer.  
The results we obtain are similar to those in 
\cite{MacGregor_Stessin}.  

By the zero-set of an $A^p$ function not identically $0$, 
we mean its collection of zeros, repeated according to multiplicity.  Such 
a set will be countable, since the zeros of analytic functions are 
discrete.  Given an $A^p$ zero set, we can consider the space $N^p$  
of all functions that vanish on that set.  More precisely, $f\in A^p$ 
is in $N^p$ if it vanishes at every point in the 
given zero set, to at least the prescribed multiplicity.  

If the zero set does not include $0$, 
we pose the 
extremal problem of finding $G \in N^p$ such that $\|G\|_{A^p} = 1$, and 
such that $G(0)$ is positive and as large as possible.  If the zero set 
has a zero of order $n$ at $0$, we instead maximize $G^{(n)}(0)$.  For 
$0<p<\infty$, this problem has a unique solution, which is called the 
canonical divisor.  For $1<p<\infty$, this follows from the fact that an 
equivalent problem is to find an $F \in N^p$ with 
$F(0)=1$ and $\|F\|_{A^p}$ as small as possible.  It is well known that 
the latter problem has a unique solution 
(see e.g.~\cite{tjf1}, Proposition 1.3). 

In this section, we discuss the problem of determining the canonical divisor 
when $p$ is an even integer, and the zero set is finite.  We show how the 
methods of this paper can be used to characterize the 
canonical divisor.   Our methods show that if $G$ is the canonical 
divisor, then $G^{p/2}$ is 
a rational function with residue $0$ at each of its poles,
which is the content of the following theorem.

\begin{theorem}\label{canonical_nec}
Let $p$ be an even integer.
Let $z_1, \ldots, z_N$ be distinct points in $\D$, and 
consider the zero-set consisting of 
each of these points with multiplicities $d_1, \ldots, d_N$, respectively. 
Let $G$ be the canonical divisor for this zero set. 
Then there are constants $c_0$ and $c_{nj}$ for $1 \le n \le N$ and 
$0 \le j \le (p/2)d_n - 1$, such that 
\[
\begin{split}
G(z)^{p/2} &= c_0 + \sum_{n=1}^N \sum_{j=0}^{(p/2)d_n - 1} 
         \frac{c_{nj}}{(1-\conj{z_n}z)^{j+2}}, 
\qquad \quad \text{if $z_n \ne 0$ for all $n$, and}\\
G(z)^{p/2} &= c_0 z^{(p/2)d_1} + \sum_{j=0}^{(p/2)d_1-1} c_{1j}z^j 
           + \sum_{n=2}^N \sum_{j=0}^{(p/2)d_n - 1} 
         \frac{c_{nj}}{(1-\conj{z_n}z)^{j+2}},
\quad \text{if $z_1 = 0.$}
\end{split}
\]
\end{theorem}
\begin{proof}
Our goal is to show that $G^{p/2}$ is the kernel for some linear combination
 of certain 
 derivative-evaluation functionals.  Because we know what the 
 kernel of any derivative-evaluation functional is, we will be able 
 to show that $G$ has the form stated in the theorem. 

Let $A_n = d_n( (p/2)-1)$.
For $1 \le n \le N$ and $0 \le j \le A_n-1$, 
let $h_{nj}$ be a polynomial 
such that 
\[ h_{nj}^{(m)}(z_k) = \begin{cases}
                     1, \text{ if $m=n$ and $k=j$ }\\
                     0, \text{ otherwise.}
\end{cases}
\]
For $f \in A^p$, define
\[
\widehat{f}(z) = f(z) - \sum_{n=1}^N \sum_{j=0}^{A_n-1} a_{nj} h_{nj}(z) 
\]
where $a_{nj} = f^{(j)}(z_n)$.  
Since $\widehat{f}$ has zeros of order $A_n=d_n( (p/2)-1)$ at each $z_n$, the function 
\[ 
\widetilde{f} = 
\frac{1}{G^{(p/2)-1}} \widehat{f}
\]
is in $A^p$.

But then
\[
\begin{split}
\int_{\D} \conj{G}^{p/2}  f\, d\sigma &= 
\int_{\D} \conj{G(z)}^{p/2} \left(
\widehat{f}(z) + \sum_{n=1}^N \sum_{j=0}^{A_n-1} a_{nj} h_{nj}(z)
\right)\,d\sigma \\
&= 
\int_{\D} |G(z)|^{p-1} \conj{\sgn{G(z)}} \widetilde{f}(z)\, d\sigma
+
\sum_{n=1}^N \sum_{j=0}^{A_n-1} a_{nj} \int_{\D} \conj{G(z)}^{p/2} h_{nj}(z) 
\,d\sigma\\
&= \mathrm{I} + \mathrm{II}.
\end{split}
\] 
Now, II is a linear combination of the numbers $a_{nj}$ for 
$1 \le n \le N$ and $0 \le j \le A_n - 1$, so we turn our attention to 
I.
The canonical divisor $G$ is a constant multiple of the function 
$F \in A^p$ of smallest norm that has zeros of order $d_n$ at each 
$z_j$ and such that $F^{(m)}(0)=1$, where $m$ is the order of the zero-set at 
$0$. 
By Theorem \ref{min_int1}, $\bp(|G|^{p-1} \sgn \conj{G})$ 
is the kernel for a linear combination of appropriate derivative evaluation 
functionals at the points $0, z_1, \cdots, z_n$. Thus, we 
have 
\[
\begin{split}
\int_{\D} |G|^{p-1}\conj{\sgn G} \widetilde{f}\, d\sigma &=
\int_{\D} \bp(|G|^{p-1}\conj{\sgn G}) \widetilde{f}\, d\sigma \\ &=
       b_0 \widetilde{f}^{(m)}(0) + 
       \sum_{n=1}^N \sum_{j=0}^{d_n-1} b_{nj} \widetilde{f}^{(j)}(z_n),
\end{split}
\]
for some complex constants $b_0$ and $b_{nj}$.  
Note that 
$\widetilde{f}(z) = G_n \widehat{f}_n$ where 
\[
G_n(z) = (z-z_n)^{A_n} \frac{1}{G^{(p/2)-1}(z)} 
\]
and 
\[
\widehat{f}_n(z) = (z-z_n)^{-A_n}
\widehat{f}(z).
\] 
Note that 
\[
\widetilde{f}^{(j)}(z_n) = \sum_{k=0}^j \binom{j}{k} 
\widehat{f}_n^{(k)}(z_n) G_n^{(j-k)}(z_n)
\]
and  
\[
\begin{split}
\widehat{f}_n^{(k)}(z_n) &= \frac{k!}{(k+A_n)!} 
\frac{d^{k+A_n}}{dz^{k+A_n}} 
\widehat{f}(z_n)\\
&=
\frac{k!}{(k+A_n)!}
\frac{d^{k+A_n}}{dz^{k+A_n}} 
\left[
f(z) - \sum_{n=1}^N \sum_{s=0}^{A_n-1} a_{ns} h_{ns}(z) 
\right].
\end{split}
\]
Thus, 
$\widetilde{f}^{(j)}(z_n)$ is a linear function of the numbers 
$a_{ns}$ and the numbers $f^{(k)}(z_n)$ for $0 \le k \le 
j + A_n.$  Recall that $a_{ns} = f^{(s)}(z_n)$.  

Also, if $m=0$, then 
\[ \widetilde{f}^{(m)}(0) = 
G(0)^{1-(p/2)} \widehat{f}(0) = 
G(0)^{1-(p/2)}
\left(f(0) - 
               \sum_{n=1}^N \sum_{j=0}^{A_n-1} a_{nj} h_{nj}(0)
\right),\]
so $\widetilde{f}^{(m)}(0)$ is a linear function of the numbers 
$a_{nj}$ and $f(0)$ = $f^{(mp/2)}(0)$.  
If $m \ne 0$, then we may assume $z_1=0$ and $m = d_1$, and then 
by the same reasoning as we used above for $\widetilde{f}^{(j)}(z_n)$, we 
see that$\widetilde{f}^{(m)}(z_1)$ is a linear function of the 
numbers $a_{nj}$ and the numbers 
$f^{(k)}(z_1)$ for $0 \le k \le d_1 + A_1 = d_1 + ((p/2)-1)d_1 = (p/2)d_1 = 
mp/2$. 
Thus, the term \[\mathrm{I} =\int_{\D} \conj{G^{(p/2)-1}} \widetilde{f}\, 
d\sigma\] 
is a linear combination of 
the numbers $f^{(k)}(z_n)$ for $0 \le k \le 
(d_n-1) + ((p/2)-1)d_n = (p/2)d_n - 1$, and the number $f^{(mp/2)}(0)$. 

Therefore, both I and II, and thus $\int_{\D} f \conj{G}^{p/2}\, d\sigma$, 
are linear combinations of the numbers
$f^{(k)}(z_n)$ for $0 \le k \le (p/2)d_n - 1$, and the number $f^{(mp/2)}(0)$. 
And thus, $G^{p/2}$ is the kernel for a derivative-evaluation functional 
depending on 
$f^{(j)}(z_n)$ for $1 \le n \le N$ and $0 \le j \le (p/2)d_n-1$,
as well as $f^{mp/2}(0)$.    
Therefore $G^{p/2}$ has the desired form.
\end{proof}

The previous theorem gave a condition on $G^{p/2}$ that must be satisfied 
if $G$ is the canonical divisor of a given zero set.  The following 
theorem says that condition, along with a few other more obviously 
necessary ones, is 
also sufficient.  
\begin{theorem}\label{canonical_nec_suf}
Let $p$ be an even integer.
Let $z_1, \ldots, z_N$ be distinct points in $\D$, 
and consider the zero-set consisting of 
each of these points with multiplicities $d_1, \ldots, d_N$, respectively. 
The canonical divisor for this zero set 
is the unique function $G$ having $A^p$ norm $1$ such that 
$G(0)>0$ (or $G^{(m)}(0) > 0$ if $G$ is required to have a zero of 
order $m$ at the origin), and such that 
$G^{p/2}$ 
has zeros of order $pd_n/2$ at each $z_n$, and 
\[
\begin{split}
G(z)^{p/2} &= c_0 + \sum_{n=1}^N \sum_{j=0}^{(p/2)d_n - 1} 
         \frac{c_{nj}}{(1-\conj{z_n}z)^{j+2}} \qquad 
           \text{if $z_n \ne 0$ for all $n$ or}\\
G(z)^{p/2} &= c_0 z^{(p/2)d_1} + \sum_{j=0}^{(p/2)d_1-1} c_{1j}z^j 
           + \sum_{n=2}^N \sum_{j=0}^{(p/2)d_n - 1} 
         \frac{c_{nj}}{(1-\conj{z_n}z)^{j+2}} \qquad
           \text{if $z_1 = 0$.}
\end{split}
\]
\end{theorem}
\begin{proof}
By Theorem \ref{canonical_nec} 
and the definition of the canonical divisor, 
the stated conditions are necessary for a function 
to be the canonical divisor. Suppose that $G$ is a function satisfying the 
stated conditions. We will prove the theorem by 
applying Theorem \ref{min_int1} to
$\bp(G^{p/2} \conj{G}^{(p/2)-1})$. 

We first discuss the proof under the assumption that $z_n \ne 0$ 
for all $n$.  
First, as above, $\bp\left(\conj{G}^{(p/2)-1}\right) = \conj{G(0)}^{(p/2)-1}.$
Now, by Proposition \ref{general_proj_zero}, 
\[
\bp\left( \frac{1}{(1-\conj{z_n}z)^{j+2}} \, \conj{G(z)}^{(p/2)-1}\right) = 
\sum_{k=0}^{j-((p/2)-1)d_n} C_{n,j,k} \frac{1}{(1-\conj{z_n}z)^{k+2}},
\]
where the constants $C_{n,j,k}$ may depend on $G$.  
But 
if  $j \le (p/2)d_n - 1$, then
$j-((p/2)-1)d_n \le d_n -1$. 
Thus 
\[
\bp\left( G^{p/2} \, \conj{G(z)}^{(p/2)-1}\right) = 
B_0 + \sum_{n=1}^N \sum_{k=0}^{d_n - 1} 
       \frac{B_{n,k}}{(1-\conj{z_n}z)^{k+2}},
\]
where $B_{n,k} = \sum_{j=k+((p/2)-1)d_n}^{(p/2)d_n - 1} c_{nj} C_{n,j,k}$ 
and $B_0 = c_0 \conj{G(0)}^{(p/2)-1}.$ 
By Theorem \ref{min_int1}, $G$ is a multiple of the canonical divisor. 
But the conditions that $G^{(m)}(0) > 0$ and $\|G\|_{A^p} = 1$ imply that 
$G$ is the canonical divisor.  

The case where $z_1 = 0$ is similar, but we also use the fact that 
$\bp(z^j \conj{G}^{(p/2)-1})$ is a polynomial of degree at most 
$j - [(p/2)-1]d_1$, or zero if $j < [(p/2)-1]d_1.$ 
\end{proof}

From previous work by MacGregor and Stessin \cite{MacGregor_Stessin}, 
a weaker form of Theorem \ref{canonical_nec} is essentially known.  
In the weaker form of 
the theorem, one only knows, in the case that no $z_n = 0$, that
\[
G(z) = c_0 + \sum_{n=0}^N \frac{b_n}{1-\conj{z_n}z} + 
\sum_{n=1}^N \sum_{j=0}^{(p/2)d_n - 1} 
         \frac{c_{nj}}{(1-\conj{z_n}z)^{j+2}} 
\]
for some constants $b_n$.  The case where $z_1=0$ is similar.  
(Although their work also gives a fairly explicit method of 
finding the canonical divisor, it does not seem to be 
clear from their results that the $b_n$ will always be zero.)
To derive 
Theorem \ref{canonical_nec_suf} from the weaker 
form of the theorem, we can use the following 
proposition, which gives another indication of why the residues of 
$G^{p/2}$ must all be zero.  It basically says that nonzero residues 
would lead to 
terms in $\bp(G^{p/2} \conj{G}^{(p/2)-1})$ that were kernels of functionals 
of the general form 
$$f \mapsto  \frac{1}{a} \int_0^{a} f(z)g(z)\, dz,$$
where $g$ is an analytic function and $a \in \D$.
But, as the proposition explains, it would then be impossible for 
$\bp(G^{p/2} \conj{G}^{(p/2)-1})$ to be the kernel of a finite linear 
combination of derivative-evaluation functionals.

\begin{proposition}
Let $g$ be analytic on $\D$ and continuous on $\overline{\D}$, 
and suppose $g$ is non-zero on 
$\partial \D.$  Let $a_n \in \D$ and $a_n \ne 0$
for $1 \le n \le N$, and 
assume that $a_n \ne a_j$ for $n \ne j.$ Let  
$b_n \in \C$ for $1 \le n \le N.$ Then if any of the $b_n$ are nonzero,
$$\bp\left(\sum_{n=1}^N \frac{b_n}{1-\conj{a_n}z}\, \conj{g(z)}\right)$$
is not the kernel for a functional that is the finite linear combination of 
derivative-evaluation functionals.
\end{proposition}
Note that as is shown in \cite{DKSS_Pac} 
(see also \cite{DKS}, \cite{Sundberg}, and \cite{D_Ap}), 
the canonical divisor of a 
finite zero set is analytic in $\overline{\D}$ and non-zero on 
$\partial \D.$  This allows the proposition to be applied to Bergman 
projections of the form 
\[
\bp\left(\sum_{n=1}^N \frac{b_n}{1-\conj{a_n}z}\, \conj{G(z)^{(p/2)-1}}\right).
\]
\begin{proof}
We know by 
Proposition \ref{kernel_for_integration} that 
 $$\bp\left(\sum_{n=1}^N \frac{b_n}{1-\conj{a_n}z} \, \conj{g(z)}\right)$$
is the kernel for the functional given by
$$f \mapsto \sum_{n=1}^N \frac{b_n}{a_n} \int_0^{a_n} f(z)g(z)\, dz.$$

Suppose this functional were a linear combination of 
derivative-evaluation functionals, which we will denote by  
$f \mapsto f^{(k)}(z_j),$ where 
$1 \le j \le J$ and $0 \le k \le K.$ 
Let $h$ be a function such that $h = gf$ for some $f \in A^p$.
For fixed $g$, the values 
$f^{(k)}(z_j)$ for $1 \le j \le J$ and $0 \le k \le K$ are linear combinations
 of the 
values 
$h^{(k)}(z_j)$, where $1 \le j \le J$ and $0 \le k \le K + r(z_j)$, and 
$r(z_j)$ is the order of the zero of $g$ at $z_j.$ 
Thus the functional defined on the space $gA^p$ by  
$$h \mapsto \sum_{n=1}^N \frac{b_n}{a_n} \int_0^{a_n} h(z)\, dz$$
must be a linear function of the values $h^{(k)}(z_j)$.  By $gA^p$, we mean 
the vector space of all functions that may be written as 
$g$ multiplied by an $A^p$ function. 
Since $g$ is analytic in  $\overline{\D}$ 
and $g$ is nonzero on $\partial \D$, any 
polynomial that has all the zeros of $g$ will be in $gA^p$.

Now for each $m$ there exists a polynomial $H_m$ such that $H_m(a_m) = 1$, but 
$H_m(a_n) = 0$ for all $n\ne m$, and such that 
$H_m^{(k)}(z_j) = 0$ for all $j$ and $k$ such that $1 \le j \le J$ and 
$1 \le k \le K + r(z_j) + 1$.
Also, we may require that $H_m'$ has all the zeros of $g$, 
and that $H_m(0) = 0.$ 
Set $h_m = H_m'$.  Then $h_m$ shares all the zeros of $g$, and so it is a 
multiple of $g$. Thus  
$$\sum_{n=1}^N \frac{b_n}{a_n} \int_0^{a_n} h_m(z)\, dz = 0,$$
since 
the left side of the above equation is a linear combination of the numbers
$h_m^{(k)}(z_j)$ for  
$1 \le j \le J$ and $0 \le k \le K + r(z_j)$, and each  
$h_m^{(k)}(z_j)=0$. 
But also, for each $m$ such that $1 \le m \le N,$ we have 
$$\sum_{n=1}^N \frac{b_n}{a_n} \int_0^{a_n} h_m(z)\, dz = 
\sum_{n=1}^N \frac{b_n}{a_n} H_m(a_n) = \frac{b_m}{a_m},$$
so each $b_m = 0.$  
\end{proof}

\begin{example}
Suppose we are given distinct points $z_1, z_2, \ldots, z_N \in \D\setminus \{0\}.$  Let $p=2M$, where 
$M$ is a positive integer.  Suppose we wish to find the canonical divisor in $A^p$ for the given set of 
points.  From the theorem, we know that
\[
G(z)^M = c_0 + \sum_{n=1}^N \sum_{m=0}^{M-1} \frac{c_{nm}}{(1-\conj{z_n}z)^{m+2}}.
\]
Then we have for $0 \le k \le M-1$ and $1 \le n \le N$ that
\[
0 = \frac{d^k}{dz^k}\left( G(z)^M \right)|_{z=z_j} = 
\frac{d^k}{dz^k}c_0 + 
\sum_{n=1}^N \sum_{m=0}^{M-1}  \frac{\conj{z_n}^k (m+1+k)!/(m+1)!}{(1-\conj{z_n}z_j)^{m+2+k}}
                                                            c_{nm}.
\]
This gives a system of $NM$ equations with $NM+1$ unknowns.  Because of the 
uniqueness of the canonical divisor, there will be a unique solution to these 
equations with $c_0 > 0$ and such that $\|G\|_{A^p} = 1$. 

\end{example}

\providecommand{\bysame}{\leavevmode\hbox to3em{\hrulefill}\thinspace}
\providecommand{\MR}{\relax\ifhmode\unskip\space\fi MR }
\providecommand{\MRhref}[2]{%
  \href{http://www.ams.org/mathscinet-getitem?mr=#1}{#2}
}
\providecommand{\href}[2]{#2}

\end{document}